\documentclass[12pt]{article}
\usepackage{amsmath}
\usepackage{amsthm}
\usepackage{amsfonts}
\usepackage{mathrsfs}
\usepackage{stmaryrd}
\usepackage{setspace}
\usepackage{fullpage}
\usepackage{amssymb}
\usepackage{breqn}
\usepackage{enumitem}
\usepackage{bbold}
\usepackage{authblk}
\usepackage{comment}
\usepackage{hyperref}
\usepackage{pgf,tikz}
\usepackage{graphicx}
\usepackage{subcaption}

\newtheorem{thm}{Theorem}[section]

\newtheorem{lem}[thm]{Lemma}

\newtheorem{prop}[thm]{Proposition}

\newtheorem{clm}[thm]{Claim}

\newcommand\ex{\ensuremath{\mathrm{ex}}}

\newcommand\cF{{\mathcal F}}
\newcommand\cG{{\mathcal G}}
\newcommand\cH{{\mathcal H}}

\newcommand\cN{{\mathcal N}}
\newcommand\cP{{\mathcal P}}

\def\lc{\left\lceil}
\def\rc{\right\rceil}

\def\lf{\left\lfloor}
\def\rf{\right\rfloor}

\newtheorem*{thm*}{Theorem}
\newtheorem*{prop*}{Proposition}
\newcommand{\ignore}[1]{}

\title{Generalized Turán problems for Berge hypergraphs}
\author{
Xiamiao Zhao \thanks{\small Department of Mathematical Sciences, Tsinghua University, Beijing 100084, China. Email:
\small zxm23@mails.tsinghua.edu.cn}\,,
\hspace{0.2em}
Xin Cheng\thanks{\small School of Mathematics and Statistics, Northwestern Polytechnical University and Xi'an-Budapest Joint Research Center for Combinatorics, Xi'an 710129, Shaanxi, P.R. China. Email:
\small \texttt{xincheng@mail.nwpu.edu.cn}.}\,, \hspace{0.2em}
D\'{a}niel Gerbner\thanks{\small Alfr\'ed R\'enyi Institute of Mathematics. Email:
\small \texttt{gerbner.daniel@renyi.hu}.}\,, \hspace{0.2em} }

\date{}

\begin{document}

\maketitle

\begin{abstract}    
    Let $\cH$ be a hypergraph and $F$ be a graph. 
    If there exists a bijection between the hyperedges of $\cH$ and the edges of $F$ such that each hyperedge contains its image, then we say that $\cH$ is a \textit{Berge copy} of $F$, and the collection of Berge copies of $F$ is denoted by Berge-$F$.
    Given $r$-graphs $\cF$ and $\cH$, the generalized hyper-Tur\'{a}n number $\ex_r(n, \cH, \cF)$ is the maximum number of copies of $\cH$ in $n$-vertex $\cF$-free $r$-graphs.

We study $\ex_r(n, \cH, \textup{Berge-}F)$. For general $\cH$, we connect this problem to counting copies of the shadow graph of $\cH$ in $F$-free graphs and obtain several exact results. In particular, we show that for any hypergraph $\cH$, if $k$ is sufficiently large, then $\ex_r(n, \cH, \textup{Berge-}K_k)$ is achieved by the balanced complete $(k-1)$-partite $r$-graph, generalizing a result of Morrison, Nir, Norin, Rza{\.z}ewski and Wesolek [\textit{Journal of Combinatorial Theory, Series B}, 162 (2023) 231--243] to the case of hypergraphs.

 We show that $\ex_r(n,K_s^r,\textup{Berge-}F)\le \ex_s(n,\textup{Berge-}F)$ and present sufficient conditions for equality. We also consider the connected generalized Tur\'{a}n number for Berge paths.

\end{abstract}



\section{Introduction}


In this paper, we use capital letters to denote graphs, and calligraphic letters to denote hypergraphs.
The only exception is that we use $K_s^r$ to denote the complete $r$-uniform hypergraph with order $s$.
We use $V(G)$ and $E(G)$ to denote the vertex set and edge set of graph $G$, respectively.
We use $v(G) = |V(G)|$ and $e(G) = |E(G)|$ to denote the number of vertices and edges of $G$, respectively.
Let $G$ be a graph and let $S$ be a subset of $V(G)$.
We use $G[S]$ to denote the induced subgraph on $S$.
The same notations apply to hypergraphs.
A fundamental result in extremal graph theory is
the Turán theorem \cite{Tu}, which states that if an $n$-vertex graph does not contain $K_k$ as a subgraph, then it cannot have more edges than the Turán graph $T(n,k-1)$, which is the $n$-vertex complete $(k-1)$-partite graph with each part of order $\lfloor n/(k-1)\rfloor$ or $\lceil n/(k-1)\rceil$. More generally, the \textit{Turán number} $\ex(n,F)$ of a graph $F$ is the largest number of edges in an $n$-vertex graph that does not contain $F$ as a not necessarily induced subgraph. This problem has attracted a lot of attention, see \cite{Fu2} for a survey.

A natural generalization is to consider hypergraphs. Given an $r$-uniform hypergraph ($r$-graph for short) $\cF$, its \textit{Turán number} $\ex_r(n,\cF)$ is the largest number of hyperedges in an $r$-uniform hypergraph that does not contain $\cF$ as subhypergraph.
Another generalization of the Tur\'an number is the following.
Given graphs $F$ and $H$, let $\ex(n, H, F)$ be the maximum number of copies of $H$ in $n$-vertex $F$-free graphs.
We call this number the \textit{generalized Tur\'{a}n number}.
The systematic study of generalized Tur\'{a}n numbers was initiated by Alon and Shikhelman \cite{alon}, see \cite{GePa} for a survey.

It is a natural idea to combine the above generalizations.
Given $r$-graphs $\cF$ and $\cH$, let $\ex_r(n, \cH, \cF)$ be the maximum number of copies of $\cH$ in $n$-vertex $\cF$-free $r$-graphs.
Generalized hypergraph Tur\'an problems have attracted some attention, see e.g. \cite{axenovich2025tur，liu2020maximum} and \cite{aa} for a survey. 

In this paper, we consider a particular class of hypergraphs.
Gerbner and Palmer \cite{B3} generalized the well-known concept of Berge cycles to general graphs.
Let $\cH$ be a hypergraph and $F$ be a graph, we say that $\cH$ is a \textit{Berge copy} of $F$ if there exists a bijection between the hyperedges of $\cH$ and the edges of $F$ such that each hyperedge contains its image. In other words, we can enlarge the edges of $F$ to obtain a copy of $\cH$. We call the copy of $F$ the \textit{core} of the Berge-$F$.
Note that for a fixed graph $F$, there are many hypergraphs that are Berge copies of $F$.
We refer to this collection of hypergraphs as $Berge$-$F$.
In this paper, we consider generalized hypergraph Tur\'{a}n problems in the case we forbid all the Berge copies of a graph. We do not deal with counting Berge copies. Note that there are multiple ways to count Berge copies of $F$, as observed in \cite{gnv}. We may count the Berge-$F$ hypergraphs, or the cores, or the pairs of cores and hypergraphs.

Let us remark that the
Tur\'{a}n problems for Berge hypergraphs and generalized Tur\'{a}n problems are very much connected. The first such connection was given by Gerbner and Palmer \cite{GP2}, who proved that $\ex(n,K_r,F)\le \ex_r(n,\textup{Berge-}F)\le \ex(n,K_r,F)+\ex(n,F)$. See Lemma~\ref{gemupa} for a strengthening of the upper bound.

For $s\geq r$ and an $s$-graph $\cG$, let $\partial_r\cG$ denote the $r$-graph consisting of all sets contained in some hyperedge of $\cG$, which is called the \textit{$r$-shadow} of $\cG$.
Given a fixed copy of $\partial_2 \cH$, let $\gamma(\cH)$ denote the number of distinct copies of $\cH$ that have this fixed copy as their shadow. For example, let $K_4^{3-}$ denote the 3-uniform 4-vertex hypergraph with 3 edges. Its shadow is a $K_4$, and a given $K_4$ is the shadow of four distinct copies of $K_4^{3-}$, thus $\gamma(K_4^{3-})=4$.



We say that an $r$-uniform hypergraph $\cH$ is \textit{nice} with respect to a graph $F$, if each edge of $\partial_2 \cH$ is contained in at least $\min \{|E(F)|, \binom{r}{2}\}$ hyperedges of $\cH$. 
Given graphs $H$ and $F$ with $\chi(H)<\chi(F)=k$, we say that $H$ is weakly $F$-Tur\'an-stable, if for any $n$-vertex $F$-free graph $G$ that contains $\ex(n,H,F)-o(n^{|V(H)|})$ copies of $H$, we can obtain a complete $(k-1)$-partite graph from $G$ by adding and deleting $o(n^2)$ edges.
We say the $H$ is \textit{$F$-Tur\'{a}n-stable} if an $n$-vertex $F$-free graph $G$ contains at least $ex(n, H, F) - o(n^{|V(H)|})$ copies of $H$, then we can obtain the $(k-1)$-partite Tur\'{a}n graph by adding and deleting $o(n^2)$ edges.
The difference between weakly Tur\'{a}n-stable and Tur\'{a}n-stable is that the complete multipartite graph is replaced by Tur\'{a}n graph.
These notions were introduced in \cite{gerb}, and were used to obtain several exact generalized Tur\'an results in \cite{gerb2,ger2,ger3}.

\begin{thm} \label{thm: sharp result}

 \textbf{(i)}   If $\cH$ is nice with respect to $F$, then $\ex_r(n,\cH,\textup{Berge-}F)=\ex(n,\partial_2 \cH,F)\gamma(\cH)$.

 \textbf{(ii)}  If $\partial_2\cH$ is weakly $K_k$-Tur\'an-stable, then for sufficiently large $n$, $\ex_r(n,\cH,\textup{Berge-}K_k)=\ex(n,\partial_2 \cH,K_k)\gamma(\cH)$.
\end{thm}

We say a graph $H$ is \textit{$K_k$-Tur\'{a}n-good} if $\ex(n,H,K_k)\neq 0$ and this extremal number is attained by the Tur\'{a}n graph $T(n,k-1)$, that is, the Tur\'{a}n graph has the maximum number of copies of $H$ for sufficiently large $n$.
This notion was introduced by Gerbner and Palmer \cite{Gerbner2022Some}.
In the same paper, Gerbner and Palmer conjectured that for every graph $H$ there exists $k_0 = k_0 (H)$ such that $H$ is $K_k$-Tur\'{a}n-good for every $k \geq k_0$.
Later, Morrison, Nir, Norin, Rza{\.z}ewski and Wesolek \cite{morrison2023every} confirmed this conjecture.
\begin{thm}[\cite{morrison2023every}]\label{thm: eventually Turan good}
  Let $H$ be a graph.
  There exists an integer $k_0 = k_0(H)$ such that $H$ is $K_k$-Tur\'{a}n-good for all $k \geq k_0$.
\end{thm}
We generalize this result to Berge hypergraphs.
We say that an $r$-uniform hypergraph $\cH$ is \textit{Berge-$K_k$-Tur\'{a}n-good}, if $\ex_r(n,\cH,\textup{Berge-}K_k)\neq 0$ and this extremal number is attained by the Tur\'an graph $T_r(n,k-1)$.
Theorem \ref{thm: eventually Turan good} together with Theorem \ref{thm: sharp result}, implies that for every $r$-graph $\cH$ and for all $k\geq k'(\cH) = \max\{k_0(\partial_2 \cH),k_1(\partial_2 \cH)\}$, $\cH$ is Berge-$K_k$-Tur\'an-good.

\begin{thm}\label{thm: eventurally Berge Turan good}
    For every $r\geq 2$, and every $r$-graph $\cH$,
    there exists an integer $k' = k'(\cH)$ such that $\cH$ is Berge-$K_k$-Tur\'an-good for all $k \geq k'$.
\end{thm}

The \textit{expansion} $F^{r+}$ of a graph $F$ is the specific $r$-uniform Berge copy that contains the most vertices, i.e., we add $r-2$ new vertices to each edge such that we add distinct vertices to distinct edges. Tur\'an problems for expansions have attracted a lot of attention, see \cite{MuVe} for a survey. Generalized Tur\'{a}n problems for expansions have been studied by Zhou, Zhao and Yuan \cite{aa}, who showed the following bound.

\begin{thm}[\cite{aa}]\label{lem0}
Let $s\geq r$ be integers and $F$ be an arbitrary graph. Then ${\rm{ex}}_r(n,K_s^{r},F^{r+})\leq \ex_s(n,F^{s+})$.
\end{thm}

We present a result analogous to Theorem \ref{lem0}, which transfers the generalized Tur\'{a}n problem for Berge hypergraphs to Tur\'{a}n problem for Berge hypergraphs.

\begin{thm}\label{main}
Let $s\geq r+1$ be integers and $F$ be an arbitrary graph. Then $\ex_r(n,K_s^r,\textup{Berge-}F)\le \ex_s(n,\textup{Berge-}F)$.
\end{thm}

For Theorem \ref{main}, in the case $r = 2$ it is a well known result that $\ex(n,K_s,F)\le \ex_s(n,\textup{Berge-}F)$.
We can strengthen it in the following way: $\ex(n,K_s,F)\le \ex_r(n,K_s^r,\textup{Berge-}F)$. Indeed, consider the $r$-graph obtained by taking the $r$-cliques of an $F$-free $n$-vertex graph $G$. This hypergraph is clearly Berge-$F$-free, and every $s$-clique in $G$ turns into an $s$-clique in this hypergraph.

Note that the reverse statement of Theorem \ref{main} obviously does not necessarily hold. For example, for sufficiently large $s$ such that $K_s^r$ contains a copy of $\textup{Berge-}F$, then $\ex_r(n,K_s^r,\textup{Berge-}F)=0$.
But obviously, when $F$ has at least two edges, $\ex_s(n,\textup{Berge-}F)\geq |E(F)|-1>0$.
However, we can prove that the reverse statement holds under certain conditions.

\begin{prop}\label{prop: reverse statement}
   \textbf{(i)} If there is an $n$-vertex $s$-uniform hypergraph $\cG$ such that $|E(\cG)|=\ex_s(n,\textup{Berge-}F)$ and $\partial_r\cG$ is Berge-$F$-free, then $\ex_r(n,K_s^r,\textup{Berge-}F)=\ex_s(n,\textup{Berge-}F)$.

   \textbf{(ii)} If there is an $n$-vertex $s$-uniform hypergraph $\cG$ such that $|E(\cG)|=(1+o(1))\ex_s(n,\textup{Berge-}F)$ and $\partial_r\cG$ is Berge-$F$-free, then $\ex_r(n,K_s^r,\textup{Berge-}F)=(1+o(1))\ex_s(n,\textup{Berge-}F)$.
\end{prop}


\begin{proof}
    Observe that every hyperedge of $\cG$ induces a copy of $K_s^r$ in $\partial_r\cG$. It follows that $\cN(K_s^r,\partial_r\cG)\ge |E(\cG)|=\ex_s(n,\textup{Berge-}F)$. Since $\partial_r\cG$ is Berge-$F$-free, we have $\cN(K_s^r,\partial_r\cG)\leq \ex_r(n,K_s^r,\textup{Berge-}F)$. Thus, in the case of \textbf{(i)},
    $\ex_r(n,K_s^r,\textup{Berge-}F)\ge\ex_s(n,\textup{Berge-}F)$. Combined with the upper bound in Theorem \ref{main}, we obtain $\ex_r(n,K_s^r,\textup{Berge-}F)=\ex_s(n,\textup{Berge-}F)$. The case of \textbf{(ii)} follows similarly.
\end{proof}

Let us mention some cases where the above Proposition \ref{prop: reverse statement} gives a sharp or asymptotic result.
If $F$ is a clique of order $k>s$, then for sufficiently large $n$, $\ex_s(n,\textup{Berge-}F)$ is equal to the number of hyperedges in $T_s(n,k-1)$, which is equal to the number of copies of $K_s^r$ in $T_r(n,k-1)$, thus we obtain a sharp result in this case. If $F$ is a graph of chromatic number $k$, we analogously obtain an asymptotic result (and exact result in several cases, according to \cite{ger}).
Let $\textup{Berge-}P_k$ denote the Berge path with $k-1$ hyperedges.
If $k\geq s+1\geq r+2$, then $\ex_r(n,K_s^r,\textup{Berge-}P_k)=\ex_s(n,\textup{Berge-}P_k)$, and we know the exact value of this function \cite{CGHMZ}. If $T_\ell$ is a tree on $\ell+1$ vertices such that the Erd\H os–S\'os conjecture holds for all subtrees of $T_\ell$, then similar results hold by \cite{B4}.
Let Berge-$M_{k+1}$ denote the Berge matching with $k+1$ hyperedges.
If $2k \geq s \geq r+1$, then then $\ex_r(n,K_s^r,\textup{Berge-}M_{k+1})=\ex_s(n,\textup{Berge-}M_{k+1})$.
The exact value of Berge matching is given by \cite{Kang2022The, Wang2025The}.

Let us consider the connected generalized Tur\'an number, i.e. the maximum number of copies of $K_s^r$ in a connected hypergraph forbidden a copy of $\textup{Berge-}F$. Observe that Theorem \ref{main} does not extend to this setting. Moreover, if we look at the proof of Theorem~\ref{main}, it goes by showing that for a Berge-$F$-free $r$-uniform hypergraph $\cH$, the $s$-uniform hypergraph $\cH'$ consisting of the $s$-cliques of $\cH$ is also Berge-$F$-free. However, we cannot extend this prof to the connected case, since
even if $\cH$ is connected, $\cH'$ need not be connected.

Let $\ex_r^{conn}(n, \cF)$ be the maximum number of hyperedges in a connected $r$-uniform hypergraph on $n$ vertices that contains no $\cF$ as a subhypergraph.
Gy\H{o}ri, Methuku, Salia, Tompkins and Vizer \cite{GMSTV} obtained bounds for $\ex_r^{conn}(n,\textup{Berge-}P_{k})$. 
Then F\"{u}redi, Kostochka and Luo \cite{FKL2} determined the exact value of $\ex_r^{conn}(n,\textup{Berge-}P_{k})$ when $k\ge 4r+1\geq 13$ and for all sufficiently large $n$.
The condition on $k$ was improved to $k \geq 2r + 14 \geq 20$ by Gy\H{o}ri, Salia and Zamora \cite{GSZ}.
Recently, Zhao, Gerbner and Zhou further improved the threshold of $k$ to $k \geq 2r+2 \geq 8$.

\begin{thm}\label{thm: conn path}
    For integers $k\geq 2r+2\geq 8$, and sufficiently large $n$, we have
    $$\ex_r^{conn}(n,\textup{Berge-}P_k)=\binom{\lf k/2 \rf-1}{r-1}(n-\lc k/2 \rc)+\binom{\lc k/2 \rc }{r}.$$
\end{thm}
We extend this to counting cliques. Note that we assume here $k\geq 2s+4$, but we believe a more careful analysis could extend this result to $k\ge 2s+2$.
\begin{thm}\label{thm: connect version}
    For integers $s> r\ge 3$, $k\geq 2s+4\ge 12$, and sufficiently large $n$, we have
    $$\ex_r^{conn}(n,K_s^r,\textup{Berge-}P_k)= \binom{\lf k/2 \rf-1}{s-1}\left( n-\lc k/2 \rc\right)+\binom{ \lc k/2 \rc}{s}.$$
\end{thm}

This paper is organized as follows.
In Section \ref{sec: sharp generalized Turan number}, we give the proof of Theorem \ref{thm: sharp result} and the proof of Theorem \ref{thm: eventurally Berge Turan good}.
In Section \ref{sec: generalized Turan number}, we give the proof of Theorem \ref{main} and the proof Theorem \ref{thm: connect version}.

\section{Proof of Theorem \ref{thm: sharp result} and Theorem \ref{thm: eventurally Berge Turan good}}\label{sec: sharp generalized Turan number}

Let us start with a result of F\"{u}redi, Kostochka and Luo \cite{Furedi2019Avoiding} which was proved independently by Gerbner, Methuku and Palmer \cite{B4}.
Before stating the result, we need some definitions.
Let $G$ be a $2$-edge-colored graph with colors red and blue.
We denote the graph spanned by the red edges by $G_{red}$ and the graph spanned by blue edges by $G_{blue}$.
Let $\cN(H, G)$ be the number of copies of $H$ in the graph $G$.
\begin{lem}\label{gemupa}
    Let $\cH$ be an $r$-uniform Berge-$F$-free hypergraph.
    Then there exists a $F$-free red-blue graph $G$ such that
    \begin{equation*}
        e(\cH) \leq e(G_{blue}) + \cN(K_r, G_{red}).
    \end{equation*}
\end{lem}

Before giving the proof of Theorem \ref{thm: sharp result}, we give a lemma which generalizes Lemma \ref{gemupa} to counting the number of copies of a hypergraph $\cH$.
We remark that \textbf{(ii)} in Lemma \ref{cele} implies the upper bound in \textbf{(i)}, but we present a separate, simpler proof of the upper bound in \textbf{(i)}, because we will use the same proof in the proof of Theorem \ref{thm: sharp result}. For that reason we add an unnecessary complication to the proof: $\min \{|E(F)|, \binom{r}{2}\}-1$ could be replaced by any constant in the proof of \textbf{(i)} of Lemma \ref{cele}, but not in the proof of \textbf{(i)} of Theorem \ref{thm: sharp result}.

\begin{lem}\label{cele}
    \textbf{(i)} For any integers $r$ and $n$, graph $F$ and $r$-graph $\cH$, $\ex(n,\partial_2 \cH,F)\gamma(\cH)\le \ex_r(n,\cH,\textup{Berge-}F)=\ex(n,\partial_2 \cH,F)\gamma(\cH)+O(n^{|V(\cH))|+2-r})$.

    \textbf{(ii)} For any integers $r$ and $n$, graph $F$ and $r$-graph $\cH$, there is an $n$-vertex $F$-free red-blue graph $G$ such that $\ex_r(n,\cH,\textup{Berge-}F)$ is at most $\cN(\partial_2\cH,G_{red})\gamma(\cH)+O(n^{|V(\cH)|-r}|E(G_{blue})|)$.
\end{lem}

\begin{proof}[Proof of Lemma \ref{cele}]
For the lower bound, consider an $F$-free $n$-vertex graph $G$, and form an $r$-graph by taking the $r$-cliques of $G$ as hyperedges. This is clearly Berge-$F$-free. Each copy of $\partial_2 \cH$ in $G$ creates $\gamma(\cH)$ copies of $\cH$ in the resulting hypergraph, completing the proof of the lower bound.

    For the upper bound in \textbf{(i)}, let $\cG$ be an $n$-vertex Berge-$F$-free $r$-graph with the largest number of copies of $\cH$. Let $\cG'$ be obtained by deleting the hyperedges of $\cG$ that contain an edge that is contained in at most $\min \{|E(F)|, \binom{r}{2}\}-1$ 
    hyperedges. This way we deleted $O(n^{2})$ hyperedges, since we can count each deleted hyperedge at least once by picking an edge $\binom{n}{2}$ ways, and then $r-2$ further vertices at most $\min \{|E(F)|, \binom{r}{2}\}-1$ ways. The number of copies of $\cH$ we deleted is $O(n^{|V(\cH))|+2-r})$, since we can count the copies of $\cH$ by picking a deleted hyperedge $O(n^{2})$ ways, and then the rest of the vertices $O(n^{|V(\cH)|-r})$ ways.

We claim that $\partial_2 \cG'$ is $F$-free. Indeed, if there is a copy of $F$ in $\partial_2 \cG'$, then we will show that this copy is the core of a Berge-$F$ in $\cG'$. If $|E(F)|\le \binom{r}{2}$, then each edge in $\partial_2 \cG'$ is contained in at least $|E(F)|$ hyperedges of $\cG$. We go through the edges of our copy of $F$ in an arbitrary order, and for each edge, we pick a hyperedge of $\cG$ containing that edge. We pick these hyperedges arbitrarily, except that we always pick a new hyperedge. This is doable, since at each edge of our copy of $F$, we have picked less than $|E(F)|$ hyperedges earlier, thus there is a new hyperedge to pick. This gives us a Berge-$F$, a contradiction. If $|E(F)|\ge \binom{r}{2}$, then each edge in $\partial_2 \cG'$ is contained in at least $\binom{r}{2}$ hyperedges of $\cG$. We consider an auxiliary bipartite graph, with part $A$ consisting of the edges in our copy of $F$, part $B$ consisting of the hyperedges of $\cG'$, and we add an edge $ab$ for $a\in A$, $b\in B$ if $a$ is a subset of $b$. A matching covering $A$ would give us a Berge-$F$. That cannot exist, thus Hall's condition is violated, there is a set $A'\subseteq A$ with the neighborhood $B'$ in $B$  having fewer than $|A'|$ vertices. The number of edges between $A'$ and $B'$ is at least $|A'|\binom{r}{2}$ and at most $|B'|\binom{r}{2}$, a contradiction.

    Clearly, each copy of $\cH$ in $\cG'$ creates a copy of $\partial_2 \cH$ in $\partial_2 \cG'$, and each copy of $\partial_2 \cH$ is counted at most $\gamma(\cH)$ times this way. Therefore, $\cN(\cH,\cG')\le \cN(\partial_2\cH,\partial_2 \cG')\gamma(\cH)\le \ex(n,\partial_2 \cH,F)\gamma(\cH)$. This proves the upper bound for \textbf{(i)}.

    Let us continue with the proof of \textbf{(ii)}. Let $\cG$ be an $n$-vertex Berge-$F$-free $r$-graph. Consider the following auxiliary bipartite graph $H$. Part $A$ consists of the edges in $\partial_2 \cG$, part $B$ consists of the hyperedges of $\cG$, and a vertex $b\in B$ is joined to a vertex $a\in A$ if $b$ contains $a$.

    Consider a largest matching $M$ in $H$. We say that a path in $H$ is an \textit{alternating path} if it starts with an edge not in $M$ and alternates between edges in $M$ and edges not in $M$. Let $B_1$ denote the set of hyperedges not covered by $M$, and let $A_2$ denote the set of vertices in $A$ that can be reached from a vertex of $B_1$ by an alternating path. Observe that the vertices of $A_2$ are each covered by $M$. We color the edges in $A_2$ red, and the other vertices of $A$ covered by $M$ are colored blue.

    The resulting red-blue graph $G$ is $F$-free. Indeed, for the edges of a copy of $F$ the matching $M$ provides the hyperedges that form a Berge-$F$.

    Consider a copy of $\cH$ in $\cG$. Consider first the copies of $\cH$ that contain a hyperedge in $H$ that is matched in $M$ to a blue edge. Such copies can be obtained by picking a blue edge first, $|E(G_{blue})|$ ways, then the corresponding hyperedge one way, then the rest of the vertices $O(n^{|V(H)|-r})$ ways.

The rest of the hyperedges in $\cG$ are either in $B_1$ or can be reached from $B_1$ by an alternating path. Therefore, in $H$ they are not joined to a vertex not covered by $M$ (since that would create a larger matching), nor to a blue edge (since that blue edge could be reached from $B_1$ by an alternating path, thus would be in $A_2$ and colored red). Therefore, each such hyperedge is joined only to red edges in $H$, i.e., its shadow is a clique in $G_{red}$. If a copy of $\cH$ consist of such hyperedges, then its shadow is a copy of $\partial_2\cH$ in $G_{red}$. Each copy of $\partial_2\cH$ is counted at most $\gamma(\cH)$ times, completing the proof.
\end{proof}

Next present the proof of Theorem \ref{thm: sharp result}, which gives some cases where the lower bound in Lemma \ref{cele} is sharp. 
We restate Theorem \ref{thm: sharp result} here for convenience.

\medskip

\noindent
\textbf{Theorem.}
 \textbf{(i)}   \textit{If $\cH$ is nice with respect to $F$, then $\ex_r(n,\cH,\textup{Berge-}F)=\ex(n,\partial_2 \cH,F)\gamma(\cH)$.}

 \textbf{(ii)}  \textit{If $\partial_2\cH$ is weakly $K_k$-Tur\'an-stable, then for sufficiently large $n$, $\ex_r(n,\cH,\textup{Berge-}K_k)=\ex(n,\partial_2 \cH,K_k)\gamma(\cH)$.
}

\begin{proof}  

    Let $\cG$ be an $n$-vertex Berge-$F$-free $r$-graph with the largest number of copies of $\cH$. As in the proof of \textbf{(i)} Lemma \ref{cele}, let $\cG'$ be obtained by deleting the hyperedges of $\cG$ that contain an edge that is contained in at most $\min \{|E(F)|, \binom{r}{2}\}-1$ hyperedges. This way we deleted $O(n^{2})$ hyperedges, since we can count each deleted hyperedge at least once by picking an edge $\binom{n}{2}$ ways, and then $r-2$ further vertices at most $\min \{|E(F)|, \binom{r}{2}\}-1$ ways. The number of copies of $\cH$ we deleted is $O(n^{|V(\cH)|+2-r})$, since we can count the copies of $\cH$ by picking a deleted hyperedge $O(n^{2})$ ways, and then the rest of the vertices $O(n^{|V(\cH)|-r})$ ways.
    If $\cH$ is nice with respect to $F$, then the deleted hyperedges
    are not contained in any copy of $\cH$, thus clearly $\cN(\cH,\cG)=\cN(\cH,\cG')$.

    We claim that $\partial_2 \cG'$ is $F$-free as in the proof of \textbf{(i)} Lemma \ref{cele}. 

    Clearly, each copy of $\cH$ in $\cG'$ creates a copy of $\partial_2 \cH$ in $\partial_2 \cG'$, and each copy of $\partial_2 \cH$ is counted at most $\gamma(\cH)$ times this way. Therefore, $\cN(\cH,\cG')\le \cN(\partial_2\cH,\partial_2 \cG')\gamma(\cH)\le \ex(n,\partial_2 \cH,F)\gamma(\cH)$. This proves the upper bound for \textbf{(i)}.

    Let us turn to the proof of \textbf{(ii)}. If $\cN(\partial_2\cH,\partial_2 \cG')\le \ex(n,\partial_2 \cH,K_k)-\gamma(\cH)\cN(\cH,\cG-\cG')$, then we are done. Otherwise, by the weakly $K_k$-Tur\'an-stable property and using that $\gamma(\cH)\cN(\cH,\cG-\cG')=O(n^{|V(\cH))|+2-r})$, there is an $n$-vertex complete $(k-1)$-partite graph $T$ that can be obtained from $\partial_2\cG'$ by adding and removing $o(n^2)$ edges.

  \begin{clm}
      We can choose $T$ such a way that for its parts $V_1,\dots, V_{k-1}$, for any $i,j\le k-1$, $v\in V_i$ cannot have in $\partial_2\cG'$ more neighbors inside $V_i$ than inside $V_j$.
  \end{clm}

  \begin{proof}[Proof of Claim]
      We pick a complete $(k-1)$-partite graph $T$ such that the number $x$ of edges of $\partial_2\cG'$ inside parts $V_i$ is as small as possible. Then $v\in V_i$ with more neighbors in $V_i$ than in $V_j$ could be moved to $V_j$ to decrease $x$, a contradiction. By the stability, we know that $x=o(n^2)$. We have to show that with this choice of $T$, the number of non-edges of $\partial_2\cG'$ between parts is $o(n^2)$. Assume otherwise and compare the number of copies of $\partial_2\cH$ in $\partial_2\cG'$ and in $T$.

      Assume first that each part $V_i$ has order $\Theta(n)$. Then we claim that each edge $uv$ of $T$ is in $\Theta(n^{|V(\cH)|-2})$ copies of $\partial_2 \cH$. Indeed, $\partial_2 \cH$ has chromatic number at most $k-1$ by the $K_k$-Tur\'an-stable property. Let us pick an arbitrary proper coloring with at most $k-1$ colors. Embed an edge of $\partial_2 \cH$ to $uv$, the corresponding color classes to the parts containing $uv$, and the other color classes to distinct parts $V_i$. For each vertex other than $u$ and $v$, we had $\Theta(n)$ choices, thus we have $\Theta(n^{|V(\cH)|-2})$ embeddings. Applying this to each edge of $T$ not in $\partial_2\cG'$, we find $\Theta(n^{|V(\cH)|})$ copies of $\partial_2\cH$, each counted at most $|E(\cH)|$ times. Therefore, there are $\Theta(n^{|V(\cH)|})$ copies of $\partial_2\cH$ in $T$, not in $\partial_2\cG'$.

      On the other hand, a copy of $\partial_2\cH$ not in $T$ contains at least one of the $o(n^2)$ edges of $\partial_2\cG'$ inside parts. These can be counted by picking such an edge and $|V(\cH)|-2$ other vertices, thus their number is $O(n^{|V(\cH)|-2})$. Therefore, there are $o(n^{|V(\cH)|})$ copies of $\partial_2\cH$ in $\partial_2\cG'$, not in $T$. This implies that  $\cN(\partial_2\cH,\partial_2 \cG')<\cN(\partial_2\cH,T)-\Theta(n^{|V(\cH)|})\le \ex(n,\partial_2 \cH,K_k)-\Theta(n^{|V(\cH)|})$. This contradicts our assumption that  $\cN(\partial_2\cH,\partial_2 \cG')> \ex(n,\partial_2 \cH,K_k)-\gamma(\cH)\cN(\cH,\cG-\cG')$, completing the proof of the claim in this case.

      Assume that some parts have size $o(n)$.
        Similar to the argument above, there are $o(n^{|V(\mathcal H)|})$ copies of $\partial_2 \mathcal H$ in $T$ that are not in $\partial_2 \mathcal G'$, and there are $o(n^{|V(\mathcal H)|})$ copies of $\partial_2 \mathcal H$ in $\partial_2 \mathcal G'$ that are not in $T$.
        For the graph $T$, move $\Theta(n)$ vertices from another part to the parts of size $o(n)$, and denote the resulting graph by $T'$.
        After this operation, $o(n^{|V(\mathcal H)|})$ copies of $\partial_2 \mathcal H$ are destroyed, while $\Theta(n^{|V(\mathcal H)|})$ copies of $\partial_2 \mathcal H$ are created.
        This implies that $
        \mathcal N(\partial_2 \mathcal H, \partial_2 \mathcal G')
        < \mathcal N(\partial_2 \mathcal H, T') - \Theta(n^{|V(\mathcal H)|})
        \leq \operatorname{ex}(n, \partial_2 \mathcal H, K_k) - \Theta(n^{|V(\mathcal H)|})$.
        This contradicts our assumption that $
        \mathcal N(\partial_2 \mathcal H,\partial_2 \mathcal G')
        > \operatorname{ex}(n,\partial_2 \mathcal H,K_k)
        - \gamma(\mathcal H)\mathcal N(\mathcal H,\mathcal G-\mathcal G')$,
        completing the proof.
  \end{proof}


    Let us return to the proof of the theorem.
    Let $U_i$ denote the set of vertices in $V_i$ that have $\Theta(n)$ non-neighbors in some other parts, and $W_i=V_i\setminus U_i$. Then clearly $|U_i|=o(n)$.

    \begin{clm}\label{klik}
        If $u\in U_i$, then $u$ has $o(n)$ neighbors in $V_i$.
    \end{clm}

    \begin{proof}[Proof of Claim]
        Assume that $u$ has $\Theta(n)$ neighbors in $V_i$, then also in each other part. Consider a neighbor $v$ of $u$ in $W_i$. We will find a copy of $K_k$ in $\partial_2 \cG'$ that contains $u,v$ and one vertex from $k-2$ other parts each. We pick these vertices one by one, from a $W_j$. Each time, we have $u$ and one vertex from each of at most $k-2$ sets $W_j$. There is some $\ell$ such that we have not picked any vertex from $V_\ell$ yet. The already picked vertices from sets $W_j$ each have $o(n)$ non-neighbors in $V_\ell$, thus all but $o(n)$ vertices of $V_\ell$ are adjacent to each. Recall that $u$ has $\Theta(n)$ neighbors in $V_\ell$, thus we have $\Theta(n)$ vertices adjacent to each vertex already picked. Out of these $\Theta(n)$ vertices, we pick one from $W_\ell$. We found a copy of $K_k$, a contradiction completing the proof of the claim.
    \end{proof}

Let us remark that the proof of the above claim also shows that there is no edge inside $W_i$, and return to the proof of the theorem. Again, we compare the number of copies of $\partial_2\cH$ in $\partial_2\cG'$ and in $T$. Each copy in $\partial_2\cG'$ that is not in $T$ contains an edge inside $V_i$ for some $i$. One endpoint of such an edge is $u\in U_i$. The number of copies of $\partial_2\cH$ in $\partial_2\cG'$ but not in $T$ that contain $u$ is $o(n^{|V(\cH)|-1})$, since we pick a neighbor of $u$ inside $V_i$ $o(n)$ ways, and then the other $|V(\cH)|-2$ vertices $O(n)$ ways each. Therefore, the total number of copies of $\partial_2\cH$ in $\partial_2\cG'$ but not in $T$ is $o(\sum_{i=1}^{k-1}|U_i|n^{|V(\cH)|-1})$.

Rather than counting the number of copies of $\partial_2\cH$ in $T$ but not in $\partial_2\cG'$, we count only those that contain some $u\in U_i$. For each such vertex, there is some $j\neq i$ such that $u$ has $\Theta(n)$ non-neighbors in $V_j$. We consider a proper $(k-1)$-coloring of $\partial_2\cH$, and embed each color class to distinct parts $V_\ell$. We make sure to embed a vertex into $u$ and one of its neighbors to a vertex in $V_j$ that is a non-neighbor of $u$ in $\partial_2\cG'$. Each time we have $\Theta(n)$ choices (except when we embed to $u$), thus we find $\Theta(n^{|V(\cH)|-1})$ embeddings of $\partial_2\cH$ in $T$ but not in $\partial_2\cG'$ that contain $u$. In total, we find  $\Theta(\sum_{i=1}^{k-1}|U_i|n^{|V(\cH)|-1})$ embeddings. Each copy is counted $O(1)$ times, thus $\cN(\partial_2\cH, T)\ge \cN(\partial_2\cH,\partial_2\cG')+\Theta(\sum_{i=1}^{k-1}|U_i|n^{|V(\cH)|-1})$. Therefore, we are done if $\sum_{i=1}^{k-1}|U_i|>0$ and $r\ge 4$.

Let $\sum_{i=1}^{k-1}|U_i|=0$ or $r=3$, and assume first that there is no hyperedge in $\cG-\cG'$ that contains at least two vertices of $W_i$ for some $i$. If $\sum_{i=1}^{k-1}|U_i|=0$, then this means that each hyperedge contains only one vertex from different parts $V_i$. Then $\cG$ is a sub-hypergraph of the complete $(k-1)$-partite $r$-graph with parts $V_i$, hence its shadow $\partial_2\cG$ is a subgraph of $T$, thus $\cN(\cH,\cG)\le \cN(\partial_2\cH,\partial_2\cG)\gamma(\cH)\le \cN(\partial_2\cH,T)\gamma(\cH)\le \ex(n,\partial_2\cH,K_k)\gamma(\cH)$ and we are done.

If $r=3$, then we have deleted $O(n^{|V(\cH)|-1})$ copies of $\cH$  from $\cG$ to obtain $\cG'$. Therefore, we are done unless $\sum_{i=1}^{k-1}|U_i|=O(1)$ and we have deleted $\Theta(n^{|V(\cH)|-1})$ copies of $\cH$ from $\cG$ to obtain $\cG'$. This second property is possible only if there were $\Theta(n^2)$ edges contained in at most $\min \{|E(F)|, \binom{r}{2}\}-1$ hyperedges. But these edges are not in $\cG'$. Since there are $o(n^2)$ edges not in $\cG'$ between parts or incident to some $U_i$, we have some of them inside some $W_i$, a contradiction.

We have obtained that there is a hyperedge $h$ containing $u_1,u_2\in W_i$. Then we find a $K_k$ as in Claim \ref{klik}. We claim that this is the core of a Berge-$K_k$. Indeed, for the edge $u_1u_2$ we use the hyperedge $h$. Since the other edges of the $K_k$ are in $\cG'$, they are contained in at least $\min\{|E(F)|, \binom{r}{2}\}$ hyperedges, thus similar to proof of Lemma \ref{cele}, we will find a copy of Berge-$K_k$, completing the proof.
\end{proof}

Finally, we show that every hypergraph $\cH$ is eventually Berge Tur\'{a}n-good.

\begin{proof}[Proof of Theorem \ref{thm: eventurally Berge Turan good}]
    Let $\cH$ be an $r$-graph and let $k \geq k'(\cH) = \max \{k_0(\partial_2 \cH),k_1(\partial_2 \cH)\}$.
    By the result in \cite{gerbner2024stability}, $\partial_2 \cH$ is $K_k$-Tur\'{a}n-stable, thus is weakly $K_k$-Tur\'{a}n-stable.
    By Theorem \ref{thm: sharp result}, we have $\ex_r(n, \cH, \textup{Berge-}K_k) = \ex(n, \partial_2 \cH, K_k) \gamma(\cH)$.
    Since $k \geq \max \{k_0(\partial_2 \cH),k_1(\partial_2 \cH)\}$, $\partial_2 \cH$ is $K_k$-Tur\'{a}n-good, which means the Tur\'{a}n graph has the maximum number of copies of $\partial_2 \cH$ for sufficiently large $n$.
    This implies that the Tur\'{a}n graph $T_r(n,k-1)$ has the maximum number of copies of $\cH$.
    Thus $\cH$ is eventually Berge Tur\'{a}n-good.
\end{proof}

\section{Proof of Theorem \ref{main} and Theorem \ref{thm: connect version}}\label{sec: generalized Turan number}

In this section, we give the proof of Theorem \ref{main}, which states that $\ex_r(n,K_s^r,\textup{Berge-}F)\le \ex_s(n,\textup{Berge-}F)$ for an arbitrary graph $F$ when $s\geq r+1$.

\begin{proof}[Proof of Theorem \ref{main}]
    Consider an $n$-vertex Berge-$F$-free $r$-uniform hypergraph $\cH$ with $\ex_r(n,K_s^r,\textup{Berge-}F)$ copies of $K_s^r$. Let $\cH'$ denote the $s$-uniform hypergraph on the same vertex set, where the hyperedges are formed by the $s$-cliques in $\cH$. We are going to show that $\cH'$ is Berge-$F$-free. This will imply that $\ex_r(n,K_s^r,\textup{Berge-}F)\le |E(\cH')|\le \ex_s(n,\textup{Berge-}F)$.

   Assume that there is a Berge-$F$ in $\cH'$ and let $F^*$ denote its core. Consider the following auxiliary bipartite graph $G$. Part $A$ consists of the edges of $F^*$, part $B$ consists of the hyperedges of $\cH$, and there is an edge between $a\in A$ and $b\in B$ if $a$ is a subset of $b$. Then each vertex of $A$ is an edge contained in a copy of $K_s^r$ in $\cH$, thus its degree in $G$ is at least $\binom{s-2}{r-2}$. Clearly, each vertex in $B$ has degree at most $\binom{r}{2}$ in $G$.

   If $s\ge r+2$, then each degree in $A$ is at least as large as any degree in $B$. Any $A'\subset A$ is incident to at least $|A'|\binom{s-2}{r-2}$ edges in $G$, thus its neighborhood has order at least $|A'|\binom{s-2}{r-2}/\binom{r}{2}\ge |A'|$. Therefore, Hall's condition holds, hence there is a matching in $G$ covering $A$. This corresponds to a Berge-$F$ in $\cH$, a contradiction.

   If $s=r+1$, then each degree in $A$ is at least $r-1$.
   Let $A'$ denote the set of vertices in $A$ with degree more than $\binom{r}{2}$.
   Since the neighborhood of $A'$ has order larger than $|A'| \binom{r}{2} / \binom{r}{2} \geq |A'|$, by Hall's condition there is a matching covering $A'$.
   We are going to extend this matching to a matching covering $A$.
   We go through the vertices $a\in A\setminus A'$ one by one.
   Each time we consider the corresponding hyperedge in $\cH'$, and pick an $r$-edge from that $(r+1)$-set as the element of $b$ adjacent to $a$.
   We have to pick one of the $r$-sets that contains $a$ and has not been picked earlier.
   
   If it is not possible, then we have a 2-set $a \in A\setminus A'$ in $\cH$ and an $(r+1)$-set $S$ containing $a$ such that each $r$-subset of $S$ that contains $a$ have already been picked.
   This means that for those $r-1$ $r$-sets, we have an outside vertex that forms an $s$-clique with them in $\cH$.
   In other words, for each vertex $u\in S \setminus a$ there is a vertex $u'\not\in S$ such that $(S\setminus u)\cup u'$ is a clique in $\cH$.
   Then for any $u,v\in S\setminus a$, both $(S\setminus \{u,v\})\cup u'$ and $(S\setminus \{u,v\})\cup v'$ are hyperedges of $\cH$.
   Moreover, $a$ is contained in $r-1$ $r$-sets in $S$.
   Hence, $a$ is contained in at least $(r-1) + 2\binom{r-1}{2} = (r-1)^2\ge \binom{r}{2}$ hyperedges, a contradiction.
   Thus we can find a matching covering $A$, which corresponds to a Berge-$F$ in $\cH$, a contradiction.
\end{proof}

Given a hypergraph $\cH'$ and a component $\cH_i$ of $\cH'$, we say that a set $B$ of vertices inside $\cH_i$ forms a \textit{hanging block} if $B$ contains a vertex $v$ such that all the hyperedges intersecting with $B\setminus \{v\}$ are contained in $B$.
We need the following result.
\begin{lem}[Khormali, Palmer \cite{khpa}]\label{lem: berge star}
    For fixed integers $\ell >r\geq 2$, if $x$ is a vertex of degree more than $\binom{\ell-1}{r-1}$, then there is a Berge-$S_{\ell}$ with center $x$.
\end{lem}

Now, we prove Theorem \ref{thm: connect version}, which determines the connected generalized Tur\'an number of a Berge path.
We rewrite it here for convenience.
\bigskip

\noindent \textbf{Theorem.}
\textit{
    For integers $s> r\ge 3$, $k\geq 2s+4\ge 12$, and sufficiently large $n$, we have}
    $$\ex_r^{conn}(n,K_s^r,\textup{Berge-}P_k)= \binom{\lf k/2 \rf-1}{s-1}\left( n-\lc k/2 \rc\right)+\binom{ \lc k/2 \rc}{s}.$$

\begin{proof}
    The lower bound is obtained by a construction.
    Fix a vertex set $S$ of size $\lfloor \frac{k}{2} \rfloor - 1$ and add $n - |S|$ additional vertices.
    Take all $r$-sets that contain at least $r-1$ vertices in $S$ as hyperedges.
    When $k$ is even, those constitute all hyperedges.
    When $k$ is odd, we fix two vertices $u_1, u_2$ outside $S$, and take all $r$-sets that contain $u_1, u_2$ together with $r-2$ vertices in $S$.
    The resulting hypergraph is Berge-$P_k$-free.
    If $k$ is even, then each $s$-set that contains at least $s-1$ vertices in $S$ forms a $K_s^r$.
    If $k$ is odd, then each $s$-set that either contains at least $s-1$ vertices in $S$ or contains $u_1, u_2$ together with $s-2$ vertices in $S$ forms a $K_s^r$.

    Let $\cH$ be the connected $n$-vertex $r$-graph with $\ex_r^{conn}(n,K_s^r,\textup{Berge-}P_k)$ copies of $K_s^r$.
    Then, we construct $\cH'$ similarly to Theorem \ref{main}, which is the $s$-graph on the same vertex set as $\cH$, where the hyperedges are formed by the $s$-cliques in $\cH$.
    The proof of Theorem \ref{main} shows that $\cH'$ is also $\textup{Berge-}P_k$-free.
    If $\cH'$ is also connected, then we are done by using the value of $\ex_s^{conn}(n,\textup{Berge-}P_k)$ in Theorem \ref{thm: conn path}.

    Let $\cH_1,\dots,\cH_h$ be the components of $\cH'$.
    Theorem \ref{main} also implies that if there is a copy of $\textup{ Berge-}F$ in $\cH_i$, then there is a copy of $\textup{Berge-}F$ in $\cH[V(\cH_i)]$ in the original hypergraph $\cH$, and with the same core.
    For each component $\cH_i$, we classify them as follows.
    \begin{itemize}
        \item A component $\cH_i$ is \textit{nice} if each vertex is contained in at least $\binom{\lf k/2 \rf -1}{s-1}$ hyperedges in $\cH'$ and contains no hanging blocks with size $\lf k/2 \rf$. 
        \item A component $\cH_i$ is \textit{strong} if after deleting the vertices with degree less than $\binom{\lf k/2 \rf}{s-1}$ and the non-cut vertices in hanging blocks with size  $\lf k/2 \rf$ one by one, the remaining hypergraph is not empty.
        \item A component $\cH_i$ is \textit{bad} if after deleting the vertices with degree less than $\binom{\lf k/2 \rf}{s-1}$ and the non-cut vertices in hanging blocks with size $\lf k/2 \rf$ one by one, the remaining hypergraph is empty.
    \end{itemize}
Then, we have the following claim.
\begin{clm}\label{clm: find path}
    If a component $\cH_i$ is nice, then for every vertex $v\in V(\cH_i)$, there is a Berge path with length at least $\lf k/2 \rf$ starting at $v$.
\end{clm}
\begin{proof} First we show that in $\cH_i$, any path that cannot be extended further at one of the endpoints has length at least $\lf k/2 \rf-1$.
Let $\cP$ be a path starting at $v$ with defining vertices $v=v_1,v_2,\dots,v_t$, and defining hyperedges $e_1,\dots,e_{t-1}$, where $\{v_i,v_{i+1}\}\subseteq e_i$.

    We prove that $t\ge \lf k/2 \rf$. Indeed, otherwise if $t\leq \lf k/2 \rf-1$, 
    since each vertex has degree at least $\binom{\lf k/2 \rf-1}{s-1}$, $v_t$ is contained in at least $\binom{\lf k/2 \rf-1}{s-1}-(t-1)>\binom{\lf k/2 \rf-2}{s-1}$ hyperedges that are not defining hyperedges of $\cP$. Then using Lemma \ref{lem: berge star}, we can find a Berge $S_{\lf k/2 \rf-1}$ with center $v_t$ using only non-defining hyperedges, thus we can find a vertex $v_{t+1}\not \in \{v_1,\dots, v_{t-1}\}$ that is joined to $v_t$ by a non-defining hyperedge. 
    

    We consider the vertex $v_{\lf k/2 \rf}$, which is contained in at least $\binom{\lf k/2 \rf-1}{s-1}-(\lf k/2 \rf-1)> \binom{\lf k/2 \rf-2}{s-1}$ hyperedges besides the defining hyperedges of $\cP$, and there is no other vertex $u$ besides $\{v_1,\dots,v_{\lf k/2 \rf-1}\}$ that is joined to $v_{\lf k/2 \rf}$ by a non-defining hyperedge.
    Thus, for every $v_i$, $i\leq \lf k/2\rf-1$, there is a non-defining hyperedge containing $v_{\lf k/2 \rf}$ and $v_i$.
   
    Then we have that for every vertex $v_i$ and $2\le i\le \lf k/2 
    \rf-1$, there is a Berge path $\cP^i$ with length $\lf k/2 \rf-1$ starting at $v_1$ and ending at $v_i$. Indeed, we go on $\cP$ from $v_1$ to $v_{i-1}$, then to $v_{\lf k/2 \rf-1}$ using a non-defining hyperedge, and then we go backwards on $\cP$ to $v_{i+1}$. 
    Thus, with the same proof, there is no vertex $u$ outside $\{v_2,\dots,v_{\lf k/2 \rf}\}$ that is joined to $v_i$ by a non-defining hyperedge for every $i=2,\dots,2k$.
    
    Suppose that there exists $u$ outside $\{v_1,\dots,v_{\lf k/2 \rf}\}$  that is contained in some defining hyperedges $e_i$. Then $v_1,\dots,v_{i-1},v_{\lf k/2 \rf},\dots,v_i,u$ forms a Berge path with length $\lf k/2 \rf$, where for the edge $v_iu$ we pick the hyperedge $e_i$.
    If there is no such $u$, then $\{v_1,\dots,v_{\lf k/2 \rf}\}$ forms a hanging block, a contradiction. 
\end{proof}
\begin{clm}
    There are no two components which are both nice or strong.
\end{clm}
\begin{proof}
    First we show that Claim \ref{clm: find path} extends to strong components. Observe that by definition, from a strong component $\cH_0$, we obtain a non-empty nice component $\cH_0'$ by deleting some vertices. If $v\in V(\cH_0')$, then we find a Berge path of length at least $\lf k/2 \rf$ starting at $v$ inside $\cH_0'$. Otherwise, we consider a shortest Berge path from $v$ to a vertex $v'$ of $\cH_0'$, and extend it with a Berge path of length at least $\lf k/2 \rf$ starting at $v'$ inside $\cH_0'$.

    Suppose that $\cH_1$ and $\cH_2$ are both nice or strong. 
    Since $\cH$ is connected, there exists a shortest Berge path (denoted by $\cP$) connecting $V(\cH_1)$ and $V(\cH_2)$ in $\cH$.
    Since $\cP$ is the shortest Berge path, the defining hyperedges of $\cP$ are not in $\cH[V(\cH_1)]\cup \cH[V(\cH_2)]$.
    Let the end defining vertex of $\cP$ in $\cH_i$ be $u_i$.
    According to Claim \ref{clm: find path}, there exists Berge paths with length at least $\lf k/2 \rf$ in $\cH_i$ starting at $u_i$ in $\cH'$. These two paths together with $\cP$ 
form a Berge path with length at least $2\lf k/2 \rf+1\ge k$, a contradiction.
\end{proof}
Let us delete the vertices with degree less than $\binom{\lf k/2 \rf -1}{s-1}$, and the non-cut vertices of hanging blocks with size $\lf k/2 \rf$  one by one.
When deleting the $\lf k/2 \rf-1$ non-cut vertices in the hanging blocks, the number of hyperedges we delete is at most $\binom{\lf k/2 \rf }{s}< (\lf k/2 \rf-1)\binom{\lf k/2 \rf-1}{s-1}$.
Thus, on average, the number of hyperedges destroyed when deleting one vertex is at most $\binom{\lf k/2 \rf-1}{s-1}-1$.
Then we are left with at most one component $\cH_i$ with $n'$ vertices.
Notice that $\cH_i$ is Berge-$P_k$-free.
If $n'\geq n_{s,k}$, then by Theorem \ref{thm: conn path}, we have $$e(\cH_i)\leq  \binom{\lf k/2 \rf-1}{s-1}\left(n'-\lc k/2 \rc\right)+\binom{ \lc k/2 \rc}{s}.$$
The total number of hyperedges in $\cH'$ is at most
$$\binom{\lf k/2 \rf-1}{s-1}\left(n'-\lc k/2 \rc\right)+\binom{ \lc k/2 \rc}{s}+(n-n')\left(\binom{\lf k/2 \rf-1}{s-1}-1\right),$$
which is at most
If $n'<n_{s,k}$, then the number of hyperedges in $\cH'$ is at most
$$\binom{n_{s,k}}{s}+\left(\binom{\lf k/2 \rf-1}{s-1}-1\right)n<\binom{\lf k/2 \rf-1}{s-1}\left(n-\lc k/2 \rc\right)+\binom{ \lc k/2 \rc}{s}.$$
The inequality holds when $n$ is large enough.
\end{proof}






\bigskip
\textbf{Funding}:
The research of Zhao is supported by the China Scholarship Council (No. 202506210250) and
the National Natural Science Foundation of China (Grant 12571372).

The research of Xin is supported by the National Natural Science Foundation of China (Nos. 12131013 and 12471334), Shaanxi Fundamental Science Research Project for Mathematics and Physics (No. 22JSZ009) and the China Scholarship Council (No. 202406290241). 

The research of Gerbner is supported by the National Research, Development and Innovation Office - NKFIH under the grant KKP-133819 and by the János Bolyai scholarship.

\end{document}